\documentclass[11pt, a4paper]{amsart}

\usepackage{a4, a4wide}
\usepackage{amssymb}
\usepackage{amsmath}
\usepackage{amsthm} \usepackage{amstext}
\usepackage{amscd}
\usepackage{latexsym}
\usepackage{graphics}
\usepackage{color}
\usepackage{enumitem}
\usepackage{tikz-cd}
\usepackage[all]{xy}
\usepackage[textsize=scriptsize]{todonotes}

\usepackage[colorlinks, pagebackref]{hyperref}
\hypersetup{
colorlinks=true,
citecolor=blue,
linkcolor=blue,
urlcolor=blue}

\makeatletter
\def\@tocline#1#2#3#4#5#6#7{\relax
\ifnum #1>\c@tocdepth 
\else
\par \addpenalty\@secpenalty\addvspace{#2}%
\begingroup \hyphenpenalty\@M
\@ifempty{#4}{%
\@tempdima\csname r@tocindent\number#1\endcsname\relax
}{%
\@tempdima#4\relax
}%
\parindent\z@ \leftskip#3\relax \advance\leftskip\@tempdima\relax
\rightskip\@pnumwidth plus4em \parfillskip-\@pnumwidth
#5\leavevmode\hskip-\@tempdima
\ifcase #1
\or\or \hskip 2em \or \hskip 2em \else \hskip 3em \fi%
#6\nobreak\relax
\dotfill\hbox to\@pnumwidth{\@tocpagenum{#7}}\par
\nobreak
\endgroup
\fi}
\makeatother

\theoremstyle{plain}

\newtheorem{theorem}{Theorem}[section]

\newtheorem{lemma}[theorem]{Lemma}
\newtheorem{corollary}[theorem]{Corollary}
\newtheorem{proposition}[theorem]{Proposition}

\theoremstyle{definition}

\newtheorem{conventions}[theorem]{Conventions}
\newtheorem{notation}[theorem]{Notation}
\newtheorem{remark}[theorem]{Remark}

\newtheorem{definition}[theorem]{Definition}

\numberwithin{equation}{section}

\def\<{\langle}
\def\>{\rangle} 
\def\-{\overline} 
\def\~{\widetilde}
\def\^{\widehat}

\def\@{\mathcal}
\def\!{\mathscr}
\def\#{\mathbb}
\def\_{\underline}

\DeclareMathOperator{\Fields}{Fields}
\DeclareMathOperator{\colim}{colim}
\DeclareMathOperator{\Spec}{Spec}

\DeclareMathOperator{\Gal}{Gal}
\DeclareMathOperator{\Cor}{Cor}

\DeclareMathOperator{\Ker}{Ker}
\DeclareMathOperator{\Coker}{Coker}
\DeclareMathOperator{\Img}{Im}

\DeclareMathOperator{\CH}{CH}

\DeclareMathOperator{\Div}{div}

\newcommand{\KM}{{\rm K}^{\rm M}} 		

\input{xy}
\xyoption{all}

\begin{document}

\title{Tame and curve-tame cohomology of $\mathbb A^1$-invariant \'etale sheaves}

\author{Sandeep S}
\address{School of Mathematics, Tata Institute of Fundamental Research, Homi Bhabha Road, Colaba, Mumbai 400005, India.}
\email{sandeeps@math.tifr.res.in}

\author{Anand Sawant}
\address{School of Mathematics, Tata Institute of Fundamental Research, Homi Bhabha Road, Colaba, Mumbai 400005, India.}
\email{asawant@math.tifr.res.in}
\date{}
\thanks{The authors acknowledge the support of SERB MATRICS grant MTR/2023/000228, India DST-DFG Project on Motivic Algebraic Topology DST/IBCD/GERMANY/DFG/2021/1 and the Department of Atomic Energy, Government of India, under project no. 12-R\&D-TFR-5.01-0500.}

\begin{abstract}
We extend the definition of the unramified curve-tame cohomology groups to $\mathbb{A}^1$-invariant \'etale sheaves under some additional hypotheses. We define a pairing of this group with the Suslin homology satisfying desirable properties and using this, we show that the unramified curve-tame cohomology of a smooth geometrically connected variety over a field of positive characteristic agrees with the cohomology of the base field. 
\end{abstract}

\maketitle
\tableofcontents

\setlength{\parskip}{4pt plus1pt minus1pt}
\raggedbottom

\section{Introduction}

The notion of homotopy invariance is the cornerstone of Voevodsky's construction of the triangulated category of motives \cite{Voevodsky-DM} in the sense that the category of homotopy invariant sheaves with transfers on smooth schemes a field is used as an essential building block.  A certain special class of homotopy invariant sheaves with transfers, called \emph{homotopy modules}, was identified with \emph{cycle modules} in the sense of \cite{Rost1996} in the Ph.D. thesis of D\'eglise \cite{Deglise} (see also \cite{Deglise-MG}).  Rost's theory of cycle modules gives an alternate approach and a generalization of classical intersection theory and can be seen as an axiomatization of fundamental properties of Milnor $K$-theory. 

One of the motivations for this work is to develop some tools needed to attack the question of Rost nilpotence for cycles having torsion primary to the characteristic $p>0$ of the base field, using a combination of the methods in \cite{Rosenschon-Sawant}, \cite{Diaz} and \cite{Gille-Surfaces}.  A precise obstruction to the Rost nilpotence principle for smooth projective varieties of dimension $\geq 3$ can be explicitly written down in terms of actions of correspondences on certain cohomology groups of \'etale motivic complexes $\#Q_\ell/\#Z_\ell(q)$ (see \cite[Remark 4.7]{Rosenschon-Sawant} and \cite[Theorem 2.4]{Diaz}), where $\ell$ runs through all the primes.  The obstruction in the case of cycles having torsion primary to the characteristic is determined by the analogues of Rost-style cycle cohomology groups associated with functors of the form $H^{q+i}(-,\#Q_p/\#Z_p(q))$, for $i=0,1$. One of the main examples of an interesting reciprocity sheaf that is not $\#A^1$-invariant is the logarithmic de Rham-Witt sheaf of Illusie \cite{Illusie}.  The importance of this example is that over the \'etale site, the logarithmic de Rham-Witt sheaf $\nu_r(q)$ in weight $q$ can be identified up to a shift with the \'etale motivic complex $\#Z/p^r\#Z(q)$ in weight $q$, by the work of Geisser and Levine \cite{Geisser-Levine-Crelle}.  Therefore, it is an interesting question to investigate whether the functors $H^{q+i}(-,\#Q_p/\#Z_p(q))$ or $H^{q+i}(-,\#Z/p^r\#Z(q))$ for $i=0$ admit a structure similar to cycle modules.  

In the case $i=0$, the above functor can be identified with Rost's cycle complex for the cycle module corresponding to mod-$p^r$ Milnor $K$-theory under the isomorphism $H^{q}_{\text{\'et}}(F, \#Z/p^r\#Z(q)) \simeq \KM_q(F)/p^r$ for any field $F$ obtained by Bloch-Gabber-Kato (see \cite{Bloch-Kato}).  In the case $i=1$, it is known that the groups $H^1(-,\nu_r(q)) = H^{q+1}(-, \#Z/p^r\#Z(q))$ do not form a cycle module as the residue map is not defined for all valuations (see \cite{Totaro}) and in fact, $\#A^1$-invariance fails.  However, we verify in Section \ref{section tame cohomology} that this \emph{partial cycle premodule data} does satisfy all the cycle premodule and cycle module axioms (see Theorem \ref{intro thm H1} below). One may attempt to define functoriality in the style of Rost on the Kato complexes associated with de Rham-Witt sheaves; however, such an attempt fails to define nontrivial pullback maps at the level of complexes. 

The first main result of the paper investigates the partial cycle module structure on the functor $H^1(-, \@G)$ for an \'etale $\#A^1$-invariant sheaf over a field of characteristic $p>0$ under additional hypotheses.  

\begin{theorem}(see Theorem \ref{theorem existence_Kato_complex} for a precise version)
\label{intro thm H1}
Let $k$ be a field of characteristic $p>0$ and let $\mathcal{G}$ be an $\mathbb{A}^1$-invariant \'etale sheaf on $Sm_k$ satisfying Conventions \ref{additional_assumptions_on_G}. Then the groups $M_i(K)=H^1_{\text{\'et}}(K,\mathcal{G}_i)$ for $i\in \mathbb{Z}$ (where $\@G_i$ denotes the $i$th twist of $\@G$) admit a partial cycle premodule structure satisfying all the cycle premodule axioms.  Consequently, for $q\ge 0$, there exists a Kato Complex
\[
C(X,\mathcal{G}_q): 0\rightarrow\bigoplus_{x\in X_{(d)}} H^1_{\text{\'et}}(k(x),\mathcal{G}_{d+q})\xrightarrow{d} \cdots \xrightarrow{d}  \bigoplus_{x\in X_{(0)}}H^1_{\text{\'et}}(k(x),\mathcal{G}_q)\rightarrow 0.
\] 
\end{theorem}

The main point in the proof of Theorem \ref{intro thm H1} is the construction of appropriate residue maps. In the case of the logarithmic K\"{a}hler differentials, Izhboldin in \cite{Izhboldin} and Totaro in \cite{Totaro} define a residue map from a certain subgroup of the first cohomology group, called the \textit{tame} subgroup. In \cite{Kato1986}, Kato defined a residue map for the logarithmic de Rham-Witt sheaves under some restrictions on the fields which imply that the tame subgroup coincides with the first cohomology group. The two definitions of residue maps agree (see Lemma \ref{lemma residues agree}). We extend the idea of Totaro to consider general $\mathbb{A}^1$-invariant \'{e}tale sheaves with transfers and define a residue map along with several functoriality properties.

As of now, the logarithmic de Rham-Witt sheaves $\nu_r(q)$ are the only examples of \'etale sheaves satisfying the hypotheses of Theorem \ref{intro thm H1} that we are aware of.  Our work can be seen as a study of the underlying abstract formalism behind the results of \cite{Totaro} and \cite{Otabe2023}. 

In \cite{Otabe2023}, Otabe defines a further refinement of the tame cohomology called the \emph{unramified curve-tame cohomology} and uses it to give a positive answer to a question of \cite{Auel-Bigazzi-Boehning}, asking if the natural map $H^1(k, \Omega^{i}_{log}) \rightarrow H^1_{\mathrm{ur}}(k(X)/k, \Omega^i_{log})$ is an isomorphism under certain conditions. To this end, Otabe defines a pairing of the unramified curve-tame cohomology with the Suslin Homology and proves an analogous result with the unramified curve-tame cohomology replacing the unramified cohomology. Specializing to the case where the two agree, he gives a positive answer to the question.

In Sections \ref{section unramified} and \ref{section unramified ct}, we extend the definition of Otabe to the case of $\mathbb{A}^1$-invariant \'{e}tale sheaves with transfers satisfying the \emph{specialization property} and whose tame unramified cohomology is $\mathbb{A}^1$-invariant. In Section \ref{section pairing}, we generalize the pairing defined by Otabe for log-K\"ahler differentials as follows (see Proposition \ref{well_definedness_of_pairing}).

\begin{proposition}
Let $X$ be a smooth geometrically connected variety over $k$. Let $\mathcal{G}$ be such that the projection map $\mathbb{A}^1_k \rightarrow \Spec k$ induces an isomorphism $H^1(k,\mathcal{G}) \xrightarrow{\sim} H^1_{\mathrm{tame, ur}}(\mathbb{A}^1_k/k, \mathcal{G})$. Then we have a pairing\[
\langle-,-\rangle: H^S_0(X) \times H^1_{\mathrm{ct, ur}}(X/k, \mathcal{G}) \rightarrow H^1(k, \mathcal{G}).
\] satisfying certain conditions.
\end{proposition}

Using this, one gets the following analogue of the main result of \cite{Otabe2023} (see Theorem \ref{curve_tame_equals_cohomology_of_base_field}).

\begin{theorem} Let $\mathcal{G}$ be an $\mathbb{A}^1$-invariant \'{e}tale sheaf with transfers over a field of characteristic $p>0$ satisfying the specialization Property \ref{specialization} and such that the canonical map $H^1(k,\mathcal{G}) \xrightarrow{\sim} H^1_{\mathrm{tame, ur}}(\mathbb{A}^1_k/k, \mathcal{G})$ is an isomorphism. Let $X$ be a smooth geometrically connected variety over $k$. Suppose that for any finitely generated field $K$ over $k$, the degree map $H^0_S(X_K) \rightarrow \mathbb{Z}$ is an isomorphism. Then the pullback $H^1(k, \mathcal{G}) \rightarrow H^1_{\mathrm{ct, ur}}(X/k,\mathcal{G}) $ is an isomorphism.
\end{theorem}

\subsection*{Conventions}

We work over a perfect field $k$.  We assume that every scheme is equidimensional, separated and of finite type over $k$.  

All fields will be assumed to be finitely generated over $k$.  Let $\Fields_k$ denote the category of finitely generated field extensions of $k$.  All valuations on a field are assumed to be of rank $1$ and of geometric type over $k$, which means that the local ring of the valuation is a regular local ring which is the localization of a height 1 prime ideal of an integral domain finitely generated over $k$.

For a field $F$, we will denote its Henselization by $F^h$ and its strict Henselization by $F^{sh}$, with respect to a separable closure $F^{sep}$. The absolute Galois group of $F$ will be denoted by $\Gamma_F:= {\rm Gal}(F^{sep}/F)$. The $i$th Milnor $K$-group of $F$ will be denoted by $\KM_i(F)$. For any $\Gamma_F$-module $M$, the Galois cohomology groups $H^i(\Gamma_F, M)$ will be denoted by $H^i(F, M)$, which is also the notation for the corresponding \'etale cohomology groups.  We will abuse the notation and denote the Galois cohomology classes and cocycles representing them by the same symbol as long as there is no confusion.

For a scheme $X$ over $k$, we write $X^{(i)}$ for the set of points of codimension $i$ on $X$ and $X_{(i)}$ for the set of points of dimension $i$ on $X$. We will write $Z_i(X)$ for the group of algebraic cycles of dimension $i$ on $X$ and $\CH_i(X)$ for the Chow group of algebraic cycles of dimension $i$ on $X$ (that is, the quotient of $Z_i(X)$ modulo rational equivalence).

\section{Logarithmic de Rham-Witt sheaves}
\label{section log drw}

Let $k$ be a field of characteristic $p > 0$ such that $\deg (k:k^p)= p^e$.  Let $X$ be a scheme of dimension $d$ over $k$.  For any integer $r>0$, let $W_r\Omega_X^\bullet$ denote the de Rham-Witt complex of $X$ defined in \cite{Illusie}.  For any integer $q = d+e$, we denote by $\nu_r(q): = W_r\Omega_{X, \rm log}^q$ the logarithmic de Rham-Witt sheaf of $X$ defined in \cite[Definition 2.6]{Shiho} to be the \'etale sheaf on $X$ defined to the image of 
\[
\left(\@O_X^\times \right)^{\otimes q} \to W_r\Omega_X^q; \quad x_1 \otimes \cdots \otimes x_q \mapsto {\rm dlog}[x_1] \wedge \cdots {\rm dlog}[x_q],
\]
where $[x_i] \in W_r \@O_X$ is the Teichm\"uller representative of $x_i$, for each $i$.  

In \cite{Kato1986}, Kato defined a family of complexes for $q\in \mathbb{Z}$ when $n \neq 1$ and $q\ge0$ when $n=1$ given by:
\begin{equation}
\label{eqn Kato complex cohomological}
C^\bullet(X,\mathbb{Z}/p^r\mathbb{Z}(q), n)\colon 0\rightarrow\bigoplus_{x\in X^{(0)}} H^n(k(x),\nu_r(q))\rightarrow\dots\rightarrow \bigoplus_{x\in X^{(d)}}H^n(k(x),\nu_r(q-d))\rightarrow 0, 
\end{equation}
under the identification $\#Z/p^r\#Z(q)[q] = \nu_r(q)$.  In the case $n=0$, it can be identified with Rost's cycle complex for the cycle module corresponding to mod-$p^r$ Milnor $K$-theory under the isomorphism $H^0_{\text{\'et}}(F, \mu_r(q)) = H^{n}_{\text{\'et}}(F, \#Z/p^r\#Z(n)) \simeq \KM_n(F)/p^r$ for any field $F$ obtained by Bloch-Gabber-Kato (see \cite{Bloch-Kato}).  

Our aim is to study the case $n=1$ above.  It is known that the functor $F \mapsto H^1(F, \nu_r(q))$ does not give rise to a cycle module \cite{Totaro}.  In fact, the functor $H^1(-, \nu_r(q))$ is not $\#A^1$-invariant on $Sm_k$ unlike $H^0(-, \mu_\ell^{\otimes q})$, where $\ell$ is coprime to $p$.  In Theorem \ref{theorem existence_Kato_complex} below, we will exhibit that this functor carries a slightly weaker structure than that of a cycle module.  In order to motivate our constructions for a general \'etale sheaf, we recall the construction of the tame subgroup of $H^1(F, \nu_r(q))$ studied in \cite{Totaro}.  In \cite[Section 4]{Totaro}, the tame cohomology group is defined as follows.

\begin{definition}\label{TotaroTamenessDefinition} Let $K \in \Fields_k$ and let $v$ be a discrete valuation on $K$ with residue field $k(v)$. Let $K_v$ be the completion of $K$ with respect to $v$. For $i \ge 0$, define the tame cohomology group  $H^{i+1,i}_{\mathrm{tame},v}(K)\subset H^1(K,\nu_r(i))$ as the inverse image of $\Ker (H^1(K_v, \nu_r(i))\rightarrow H^1(K_{v,\mathrm{tame}}, \nu_r(i)))$ under the restriction map $H^1(K,\nu_r(i))\rightarrow H^1(K_v,\nu_r(i))$, where $K_v$ is the completion of $K$ with respect to $v$.
\end{definition}

As pointed out in \cite{Otabe2023}, the maximal tamely ramified extension $K_{\mathrm{tame}}$ in the above definition may be replaced by the maximal unramified extension $K_{ur}$ by \cite[Remark 3.7]{Auel-Bigazzi-Boehning}.

\begin{lemma}\label{equivalenceoftameness}
If $K$ is complete and $i\ge 0$, then $H^{i+1,i}_{\mathrm{tame},v}(K)\cong H^1(k(v), H^0(K^{sh}, \nu_r(i)))$, where $K^{sh}$ is the strict Henselization of $K$ with respect to some separable closure of $k(v)$.
\end{lemma}

\begin{proof}

Recall that $\Gal(K^{sh}/K) \cong \Gal(k(v)^{sep}/k(v))$. Therefore, for the extension $K \rightarrow K^{sh}$, we have the Hochschild-Serre spectral sequence
\[
E^{a,b}_2=H^{a}(k(v), H^{b}(K^{sh}, \nu_r(i))) \Rightarrow H^{a+b}(K, \nu_r(i)).
\]
Since $E^{a,b}_2=0$ whenever $a$ or $b$ is greater than $1$, we have the following exact sequence
\[
0\rightarrow H^{1}(k(v), H^{0}(K^{sh}, \nu_r(i))) \rightarrow H^1(K, \nu_r(i))) \xrightarrow{f} H^{0}(k(v), H^{1}(K^{sh}, \nu_r(i))) \rightarrow 0.
\]
Note that $H^{0}(k(v), H^{1}(K^{sh}, \nu_r(i)))$ is the subgroup of $H^1(K^{sh}, \nu_r(i))$ invariant under the action of $\Gal(K^{sh}/K)$. The composition of $f$ with the inclusion is the restriction map $H^1(K, \nu_r(i))\rightarrow H^1(K^{sh}, \nu_r(i))$.

Since $K$ is complete and in particular Henselian, $K^{sh}$ is the maximal unramified extension $K_{ur}$ with respect to $v$ (see \cite[pg. 414]{Fu2011}). Therefore, we have
\[
H^{1}(k(v), H^{0}(K^{sh}, \nu_r(i))) \cong \Ker (H^1(K, \nu_r(i))\rightarrow H^1(K^{sh}, \nu_r(i))) = H^{i+1,i}_{\mathrm{tame},v}(K).
\]
\end{proof}

In \cite[Section 4]{Totaro}, Totaro defines a residue map $\partial'_v: H^{i+2, i+1}_{\mathrm{tame}, v}(K) \rightarrow H^1(k(v), \nu_1(i))$, whose definition we recall as follows.  We have an exact sequence
\[
0 \rightarrow H^0(K, \nu_1(i+1)) \rightarrow \Omega^{i+1}_{K} \xrightarrow{\mathcal{P}} \Omega^{i+1}_{K}/d\Omega^{i}_{K} \xrightarrow{\delta} 
H^1(K, \nu_1(i+1)) \rightarrow 0,\] where $\mathcal{P}$ is the Artin-Schreier map given by
\[
\mathcal{P}(a(db_0/b_0)\wedge \dots \wedge (db_i/b_i)) = (a^p-a)(db_0/b_0)\wedge \dots \wedge (db_i/b_i)
\]
and $\delta$ is the connecting map. By \cite[Theorem 4.4]{Totaro}, the subgroup $H^{i+2,i+1}_{\mathrm{tame},v}(K)$ is generated by elements of the form $[a, b_0, \dots, b_i\} := \delta(a(db_0/b_0)\wedge \dots \wedge (db_i/b_i))$, where $v(a) \ge 0$. The residue map is defined as 
\begin{equation}
\label{eqn Totaro residue}
\partial'_v([a,b_0,\dots,b_i\}) = [\overline{a}, \partial_v\{b_0,\dots,b_i\}.
\end{equation}
In Section \ref{section tame cohomology}, we extend the definition of the residue map to the case where the logarithmic de Rham-Witt sheaves are replaced by an arbitrary $\#A^1$-invariant \'etale sheaf with some additional hypotheses motivated by the case of logarithmic de Rham-Witt sheaves.

\section{The tame cohomology group }
\label{section tame cohomology}

Let $Sm_k$ be the category of smooth separated schemes over a perfect field $k$ of characteristic $p>0$.  All fields considered will be finitely generated over $k$. For a field $K$, let $Val(K/k)$ denote the set of rank 1 discrete valuations of geometric type on $K$ over $k$. For a valuation $v$ on a field $K$, $K_v$ will denote the completion of $K$ with respect to $v$.

\begin{notation}
\label{notation G} 
For the rest of this section, we use the following notation:
\begin{enumerate}
\item  $\mathcal{G}$ is an $\mathbb{A}^1$-invariant \'etale sheaf with transfers. By \cite[Definition 3.4.3]{Deglise}, there exist twists $\mathcal{G}_i$ for negative integers $i$.  As in \cite[Definition 3.4.17]{Deglise}, one can define the twists $\mathcal{G}_i$ for $i \in \#N$.

\item For $i\in \mathbb{Z}$ and $K\in \Fields_k$, we denote $M_i(K)\colon=H^1_{\text{\'et}}(K,\mathcal{G}_i)$.

\end{enumerate}
\end{notation}

We can define the cycle-premodule data $\mathbf{(D1)}-\mathbf{(D3)}$ for $\{M_i\}_{i\in \mathbb{Z}}$. However the residue map \textbf{(D4)} is not defined in general.

\begin{enumerate}[label={$\mathbf{(D\arabic*)}$}]
\item For an extension $\varphi\colon K\rightarrow L$ in $\Fields_k$, there are restriction maps $\varphi_*\colon M_i(K)\rightarrow M_i(L)$.

\item For a finite field extension $\varphi:K\rightarrow L$ in $\Fields_k$, there are corestriction maps $\varphi^*\colon M_i(L)\xrightarrow{}M_i(K)$ defined as follows:

Choose models $X$ and $X'$ for $K$ and $L$ respectively and a finite dominant morphism $f:X'\xrightarrow{}X$ such that it induces $\varphi$ on the fraction fields. The transpose of $f$ induces a morphism of sheaves $$f_*\mathcal{G}_{X'}\xrightarrow{}\mathcal{G}_{X}$$ on $X_{\text{\'et}}$. Applying $H^1_{\text{\'et}}(X,-)$, we get a map $$H^1_{\text{\'et}}(X,f_*\mathcal{G}_{X'})\xrightarrow{}H^1_{\text{\'et}}(X,\mathcal{G}_{X}).$$ We have such maps for each $f^{-1}(U)\xrightarrow{}U$ for every open subset $U\subset X$. Now, since $f$ is finite, $R^1f_*\mathcal{G}_{X'}=0$ (see \cite[\href{https://stacks.math.columbia.edu/tag/03QP}{Tag 03QP}]{stacks-project}); therefore, by the Leray spectral sequence, $H^1_{\text{\'et}}(X,f_*\mathcal{G}_{X'})\cong H^1_{\text{\'et}}(X',\mathcal{G}_{X'})$.

This gives us a map $H^1_{\text{\'et}}(X',\mathcal{G}_{X'})\xrightarrow{}H^1_{\text{\'et}}(X,\mathcal{G}_{X})$. Open subsets of the form $f^{-1}U$ for open subsets $U$ of $X$ are cofinal among the open subsets of $X'$. So we can define the corestriction map as
\begin{align*}\varphi^*\colon H^1_{\text{\'et}}(L,\mathcal{G})\cong \colim_{\emptyset\ne U'\subset X'}H^1_{\text{\'et}}(U',\mathcal{G})\cong\colim_{\emptyset\ne U\subset X}H^1_{\text{\'et}}(f^{-1}(U),\mathcal{G})\\
\xrightarrow{}\colim_{\emptyset\ne U\subset X}H^1_{\text{\'et}}(U,\mathcal{G})\cong H^1_{\text{\'et}}(K,\mathcal{G}).
\end{align*}

\item For $K\in \Fields_k$, there is an action $\KM_i(K)\times M_j(K)\xrightarrow{}M_{i+j}(K)$ induced by the cup product in Galois cohomology and the action 
\[
\KM_i(K)\times \mathcal{G}_{j}(K)\xrightarrow{}\mathcal{G}_{i+j}(K)\] 
defined in \cite[Definition 5.5.17]{Deglise}.  We will use $\cdot$ to denote the product as well as the action of Milnor $K$-theory groups.
\end{enumerate}

The data $\mathbf{(D1)}$--$\mathbf{(D3)}$ given above satisfies the cycle premodule axioms $\mathbf{(R1)}$--$\mathbf{(R3)}$ of \cite[Definition 1.1]{Rost1996}.  These are easy to verify and are left to the reader.  

\begin{enumerate}[label={$\mathbf{R1\alph*.}$}]
\item For field extensions $\varphi\colon K\xrightarrow{}L$ and $\psi\colon K\xrightarrow{}F$, we have $(\psi \circ \varphi)_*=\psi_* \circ \varphi_*$.

\item For finite field extensions $\varphi\colon K\xrightarrow{}L$ and $\psi\colon L\xrightarrow{}F$, we have $(\psi \circ \varphi)^*=\varphi^* \circ \psi^*$.

\item For field extensions $\varphi\colon K\xrightarrow{}L$ and $\psi\colon K\xrightarrow{}F$, where $\varphi$ is finite and $R=F\otimes_{K}L$, we have \[\psi_* \circ \varphi^*=\sum_{p\in\Spec R}l(R_p)\varphi^*_p(\psi_z)_*,\] where $\varphi_p$ is the extension $K\xrightarrow{}L\xrightarrow{}R\xrightarrow{}R/p$ and $\psi_z$ is similarly defined.
\end{enumerate}

\begin{enumerate}[label={$\mathbf{R2\alph*.}$}]
\item For a field extension $\varphi\colon K\xrightarrow{}L$, $\alpha\in K^M_n(K)$ and $\rho\in M_i(K)$, we have $\varphi_*(\alpha \cdot \rho)=\varphi_*(\alpha)\cdot \varphi_*(\rho)$.

\item If $\varphi\colon K\xrightarrow{}L$ is a finite extension and $\mu\in M_i(L)$, then $\varphi^*((\varphi_*\alpha) \cdot\mu)=\alpha \cdot \varphi^*(\mu)$.

\item If $\varphi\colon K\xrightarrow{}L$ is a finite extension and $\beta\in \KM_n(L)$, then  $\varphi^*(\beta \cdot\varphi_*(\rho))=\varphi^*(\beta)\cdot(\rho)$.
\end{enumerate}

The residue map, i.e, datum \textbf{(D4)} is not defined in general. However, we can define a residue map from a subgroup of $H^1(K, \mathcal{G}_i)$ called the \textit{tame cohomology group}. This notion was originally defined for the sheaf of K\"ahler differentials by Izhboldin in \cite{Izhboldin} and developed by Totaro in \cite{Totaro}.  We can generalize the definition of the tame subgroup for any $\mathbb{A}^1$-invariant \'etale sheaf with transfers as follows.

\begin{definition}
Let $K$ be a field and $v\in Val(K/k)$ with residue field $k(v)$. Let $\mathcal{F}$ be an $\mathbb{A}^1$-invariant sheaf over $k$. We define the tame cohomology group  $H^1_{\mathrm{tame},v}(K, \mathcal{G}_{i})\subset H^1(K,\@G_i)$ as the inverse image of \[\Ker (H^1(K_v,\mathcal{G}_{i})\rightarrow H^1(K_v^{sh}, \mathcal{G}_{i})) = \Img(H^1(k(v), H^0(K_v^{sh}, \mathcal{G}_{i})) \hookrightarrow H^1(K_v, \mathcal{G}_{i}))\]  under the restriction map $H^1(K,\mathcal{G}_{i})\rightarrow H^1(K_v,\mathcal{G}_{i})$. For a general field $K$, define \[H^1_{\mathrm{tame}}(K/k, \mathcal{G}_{i}) = \bigcap_{v\in Val(K/k)}H^1_{\mathrm{tame},v}(K, \mathcal{G}_{i}).\] 
\end{definition}

By Lemma \ref{equivalenceoftameness}, when $\mathcal{G}_{i} = \Omega^i$, this agrees with Definition \ref{TotaroTamenessDefinition}. We now define a generalization of Totaro's residue map and will show that it agrees with Totaro's definition when $\mathcal{G}_{i} = \Omega^i$. 

\begin{definition}\label{tameResidue}
Define the residue map $\partial_v$ as the composite
\begin{align*}
H^1_{\mathrm{tame},v}(K, \mathcal{G}_{i+1}) &\rightarrow H^1_{\mathrm{tame},v}(K_v, \mathcal{G}_{i+1}) \cong  H^1(k(v), H^0(K_v^{sh}, \mathcal{G}_{i+1}))\\
&\xrightarrow{H^1(k(v),\partial_v)} H^1(k(v),H^0(k(v)^{sep},\mathcal{G}_{i})) \cong H^1(k(v), \mathcal{G}_{i}),
\end{align*} 
where the two isomorphisms are given by the Hochschild-Serre spectral sequence.  This map restricts to give a residue map $\partial_v:H^1_{\mathrm{tame}}(K/k, \mathcal{G}_{i+1})\rightarrow H^1(k(v), \mathcal{G}_{i}),$ which we take as the datum $\mathbf{(D4)}.$
\end{definition}

\begin{remark}
\label{remark cocycles}
We will freely use the description of elements of $M_i(K) = H^1(K, \mathcal{G}_i)$ in terms of 1-cocycles.  If $K$ is complete, then $H^1_{\mathrm{tame},v}(K,\mathcal{G}_{i+1}) \cong H^1(k(v),H^0(K^{sh},\mathcal{G}_{i+1}))$ and the residue map in terms of 1-cocycles representing $H^1(k(v),H^0(K^{sh},\mathcal{G}_{i}))$ is given by: 
\[
\partial_v(\alpha)(\sigma)=\partial_v(\alpha(q^{-1}\sigma)),
\] 
where $\sigma\in \Gamma_{k(v)}$, $q$ is the quotient map $\Gamma_F \xrightarrow{} \Gal(F^{sh}/F)\cong\Gamma_{k(v)}$ and the $\partial_v$ is the residue map on the cycle module $H^0(-,\mathcal{G}_{\ast}).$

We will use the explicit description of the corestriction map in $\mathbf{(D2)}$ in terms of cocycles \cite[Chapter 1, Section 5.4]{Neukirch-Schmidt-Wingberg}. If $H$ is an open subgroup of a profinite group $G$ and $A$ is a $G$-module, the corestriction map $H^1(H,A)\xrightarrow{}H^1(G,A)$ is as follows: let $\alpha\colon H\xrightarrow{}A$ be a cocycle; then the corestriction of $\alpha$ to $G$ sends $\sigma\in G$ to \[\sum_{\tau\in H\backslash G}s(\tau)^{-1}\alpha(\overline{s(\tau)\sigma s(\tau\sigma)^{-1}}),\] where $s$ is a set-theoretic splitting of $G\xrightarrow{}H\backslash G$.
\end{remark}

\begin{lemma}
\label{lemma residues agree}
When $r=1$ and $\@G_i =  \Omega^i = \nu_1(i)$, the residue map $\partial_v$ agrees with Totaro's residue map $\partial'_v$ defined in \eqref{eqn Totaro residue}.
\end{lemma}

\begin{proof}
We may assume that $K$ is complete. 
Notice that any $\alpha := [a, b_0, \dots, b_i\} \in H^{i+2,i+1}_{\mathrm{tame},v}(K)$ is the cup product $[a] \smallsmile \{b_0,\dots, b_i\}$, where $[a]\in H^1(K, \mathbb{Z}/p\mathbb{Z})$ is the class of $a$ under the isomorphism $H^1(K, \mathbb{Z}/p\mathbb{Z}) \cong K/\mathcal{P}(K)$ and $\{b_0,\dots, b_i\} \in H^0(K, \nu_1(i+1)).$ As a cycle, $\alpha$ is given by $\alpha(\sigma) = [a](\sigma)(db_0/b_0)\wedge\dots\wedge(db_i/b_i)$ for $\sigma \in \Gamma_K$. Since $\alpha$ lies in the image of $H^1(k, H^0(K^{sh}, \nu_r(i+1))) \rightarrow H^1(k(v), \nu_r(i+1))$, it factors through the quotient $\Gamma_K \rightarrow \Gamma_{k(v)}$. Therefore, $\partial_v(\alpha)\overline{(\sigma)} = \overline{[a](\sigma)}\partial_v\{b_0,\dots, b_i\}$.

Therefore, to prove the lemma, it suffices to show that $[\overline{a}](\overline{\sigma}) = \overline{[a](\sigma)}$. We can explicitly compute $[a](\sigma)$ using the Artin-Schreier sequence; let $a' \in K^{sep}$ be such that $\mathcal{P}(a')=a$. Then $[a](\sigma) = \sigma(a')-a'$. By the assumptions on $\alpha$, we may take $a'\in K^{sh}$ with $v(a')\ge 0$. Then $\mathcal{P}(\overline{a'}) = \overline{a}$ so that $\overline{[a](\sigma)} = \overline{\sigma(a') - a'} = \sigma(\overline{a'}) - \overline{a'} = [\overline{a}](\overline{\sigma}).$
\end{proof}

We now show that the restriction, corestriction, and multiplication by Milnor K-theory on $H^1(-,\mathcal{G}_{i})$ restrict to $H^1_{\mathrm{tame}}(-/k, \mathcal{G}_{i}).$ From this it follows automatically that the cycle pre-module axioms {\bf (R1a-c)} and {\bf (R2a-c)} hold for the groups $H^1_{\mathrm{tame}}(-/k,\mathcal{G}_i)$.

\begin{lemma}
\label{lemma res}
Let $K\rightarrow L$ be a field extension over $k$. The restriction map $\phi_*\colon H^1(K,\mathcal{G}_{i}) \rightarrow H^1(L,\mathcal{G}_{i})$ restricts to a map $H^1_{\mathrm{tame}}(K/k, \mathcal{G}_{i}) \rightarrow H^1_{\mathrm{tame}}(L/k, \mathcal{G}_{i})$, making $H^1_{\mathrm{tame}}(-/k,\mathcal{G}_{i})$ into a functor on the category of finitely generated fields over $k$.
\end{lemma}

\begin{proof}
The proof is exactly analogous to \cite[Lemma 4.6]{Otabe2023}.
\end{proof}

\begin{lemma}
\label{lemma cores}
Let $\phi:K\rightarrow L$ be a finite field extension over $k$. Then the co-restriction map $\phi^*:H^1(L,\mathcal{G}_{i}) \rightarrow H^1(K,\mathcal{G}_{i})$ restricts a map $H^1_{\mathrm{tame}}(K/k, \mathcal{G}_{i}) \rightarrow H^1_{\mathrm{tame}}(L/k, \mathcal{G}_{i})$.
\end{lemma}

\begin{proof}
It suffices to show that for any valuation $v$ on $K$ and any valuation $w$ on $L$ restricting to $v$, the co-restriction map $H^1(L_w,\mathcal{G}_{i}) \rightarrow H^1(K_v,\mathcal{G}_{i})$ restricts to a map $H^1_{\mathrm{tame},v}(K_v/k, \mathcal{G}_{i}) \rightarrow H^1_{\mathrm{tame},w}(L_w/k, \mathcal{G}_{i})$. This follows from the following commutative diagram, which in turn follows from the functoriality of the Hochschild-Serre spectral sequence with co-restriction maps. \[\begin{tikzcd}
{H^1(L_w,\mathcal{G}_{i})} & {H^1(K_v,\mathcal{G}_{i})} \\
{H^1(k(w),H^0(L^{sh}_w,\mathcal{G}_{i}))} & {H^1(k(v),H^0(K^{sh}_v,\mathcal{G}_{i})).}
\arrow[from=1-1, to=1-2]
\arrow[hook, from=2-1, to=1-1]
\arrow[from=2-1, to=2-2]
\arrow[hook, from=2-2, to=1-2]
\end{tikzcd}\]
\end{proof}

\begin{lemma}
\label{lemma K-action}
The action of Milnor K-theory $\KM_i(K)\times H^1(K, \mathcal{G}_j) \rightarrow H^1(K,\mathcal{G}_{i+j})$ restricts to an action $\KM_i(K)\times H^1_{\mathrm{tame}}(K/k, \mathcal{G}_j) \rightarrow H^1_{\mathrm{tame}}(K/k,\mathcal{G}_{i+j})$.
\end{lemma}

\begin{proof}
Let $v\in Val(K/k)$ and let $\alpha \in H^1_{\mathrm{tame},v}(K, \mathcal{G}_j)$ and $\theta \in \KM_i(K)$. We may assume that $v$ is complete. Then $\theta\cdot\alpha$ is the cocycle sending $\sigma \in \Gamma_K$ to $\theta\cdot\alpha(\sigma)$. Now, since $\alpha$ is in the tame subgroup, it is the image of some $\beta \in H^1(k(v), H^0(K^{sh},\mathcal{G}_{j}))$. So if we let $\beta' \in H^1(k(v), H^0(K^{sh},\mathcal{G}_{i+j}))$ be the cocycle sending $\overline{\sigma} \in \Gamma_{k(v)}$ to $\theta\cdot\beta(\overline{\sigma})$, then $\sigma\cdot\alpha$ is the image of $\beta'$. Therefore, $\sigma\cdot\alpha \in H^1_{\mathrm{tame}}(K/k,\mathcal{G}_{i+j})$
\end{proof}

We now verify the axioms \textbf{R3(a-e)} below.  This may be well-known to experts in the case of logarithmic de Rham-Witt sheaves.

\begin{proposition}
\label{prop R3}

Let $\varphi: K \rightarrow L$ be a field extension in $\Fields_k$ and let $w$ be a valuation on $L$ restricting to a valuation $v$ on $K$. Let $\overline{\varphi}: k(v) \rightarrow k(w)$ be the induced extension of residue fields and $i\in \mathbb{N}$. The following relations hold.
\begin{enumerate}[label=$\mathbf{R3\alph*}$.]

\item If $v$ is a nontrivial valuation with ramification index $e$ and $\partial_v: H^1_{\mathrm{tame}}(K/k, \mathcal{G}_{i+1}) \rightarrow H^1(k(v),\mathcal{G}_{i} )$ and $\partial_w: H^1_{\mathrm{tame}}(L/k, \mathcal{G}_{i+1})  \rightarrow H^1(k(w), \mathcal{G}_{i})$ are the residue maps, and $\varphi_*$ and $\-{\varphi}_*$ are the restriction maps, then
\[
\partial_w \circ \varphi_* = e \cdot \-{\varphi}_* \circ \partial_v. 
\]

\item Assume that $\varphi$ is a finite extension and let $\partial _{v'}: H^1_{\mathrm{tame}}(K/k, \mathcal{G}_{i+1}) \rightarrow H^1(k(v'), \mathcal{G}_{i})$ be the residue maps. Then we have 
\[
\partial_v\circ\varphi^*=\sum_{v'}\varphi_{v'}^* \circ \partial_{v'},
\] 
where $v'$ runs through the extensions of $v$ to $L$ and $\-\varphi_{v'}: k(v) \to k(v')$ denotes the induced extension of residue fields. 

\item If $w$ is trivial on $K$ (that is, $v$ is trivial), with the same notation as in $\mathbf{R3a}$, we have 
\[
\partial_w \circ \varphi_* = 0. 
\]

\item Suppose $w$ is trivial on $K$ and let $\-\varphi: K \to k(w)$ denote the induced map. For any uniformizer $\pi$ for $w$, let $s_w^\pi: H^1_{\mathrm{tame}}(L/k, \mathcal{G}_{i+1}) \rightarrow H^1(k(w),\mathcal{G}_{i})$ be the specialization map. Then we have
\[
s_w^\pi \circ \varphi_* = \-\varphi_*. 
\]

\item For a unit $u$ with respect to $v$ and for any $\alpha \in H^1_{\mathrm{tame}}(K/k, \mathcal{G}_{i})$, we have
\[
\partial_v(\{u\} \cdot \alpha)=-\{\-u\} \cdot \partial_v(\alpha).
\]
\end{enumerate}
\end{proposition}

\begin{proof}
By the results of \cite{Deglise}, the functors $K \mapsto H^0(K,\mathcal{G}_{\ast})$ form a cycle module.  This fact will be used repeatedly.
We assume that $K$ and $L$ are complete with respect to $v$ and $w$ respectively. The general case follows easily from this.
We first prove $\mathbf{R3a}$.  Note that we have commutative diagrams
\[
\begin{xymatrix}{
{H^1_{\mathrm{tame},v}(K,\mathcal{G}_{i+1})} \ar[r] \ar[d]^-{\varphi_*}  & {H^1(k(v),H^0(K^{sh},\mathcal{G}_{i+1}))} \ar[d]^-{H^1(\varphi_*)}\\
{H^1_{\mathrm{tame},w}(L,\mathcal{G}_{i+1})} \ar[r]  & {H^1(k(w),H^0(L^{sh},\mathcal{G}_{i+1}))} }
\end{xymatrix}
\]
and 
\[
\begin{xymatrix}{ 
{H^1(k(v),H^0(K^{sh},\mathcal{G}_{i+1}))} \ar[r] \ar[d]^-{H^1(\varphi_*)} & {H^1(k(v),H^0(k(v)^{sep},\mathcal{G}_{i}))} \ar[d]^-{e \cdot \-\varphi_*} \\
{H^1(k(w),H^0(L^{sh},\mathcal{G}_{i+1}))} \ar[r] & {H^1(k(w),H^0(k(w)^{sep},\mathcal{G}_{i}))}}, 
\end{xymatrix}
\]
where the commutativity of the latter diagram follows from that of the corresponding diagram for $H^0(-, \mathcal{G}_{\ast})$. This implies $\mathbf{R3a}$. 
We prove $\mathbf{R3b}$ separately in Lemma \ref{corestrictions-residues}. We next prove $\mathbf{R3c}$.   Let $\alpha \in H^1_{\mathrm{tame},v}(K, \mathcal{G}_i)$ and consider a $1$-cocycle representing $\alpha$.  By Remark \ref{remark cocycles}, we have 
\[
\partial_v \circ\varphi_*(\alpha)(\sigma)= \partial_v(\varphi_*\alpha(q^{-1}\sigma))=\partial_v(\varphi_*(\alpha(\tilde{\varphi}(q^{-1}\sigma))))=0,
\]
by the corresponding result for $H^0(-, \mathcal{G}_{\ast})$.  Now, suppose that $v$ is trivial on $E$ and let $\alpha \in M_i(E)$. We prove $\mathbf{R3d}$ by a similar computation involving cocycles as above.  For any $\sigma\in \Gamma_{k(v)}$,  
\[
\begin{split}
s^{\pi}_v \circ \varphi_* (\alpha) (\sigma) = \partial_v((\{-\pi\}\cdot \varphi_*\alpha)(\pi^{-1}\sigma)) &= \partial_v(\{-\pi\} \cdot \varphi_*(\alpha(\tilde{\varphi}(\pi^{-1}\sigma)))) \\
&= s^{\pi}_v\varphi_*(\alpha(\tilde{\varphi}(\pi^{-1}\sigma)))\\
&=\overline{\varphi}_*(\alpha(\tilde{\varphi}(\pi^{-1}\sigma)))=\overline{\varphi}_*(\alpha)(\sigma).
\end{split}
\]
For a unit $u$ with respect to $v$, we have
\[
\begin{split}
\partial_v(\{u\}\cdot\alpha) (\sigma) &= \partial_v(\{u\}\cdot \alpha(\pi^{-1}\sigma))=-\{\overline{u}\}\cdot \partial_v(\alpha(\pi ^{-1}(\sigma)))=-\{\overline{u}\} \cdot \partial_v(\alpha)(\sigma)\\
&= -\{\overline{u}\} \cdot \partial_v(\alpha) (\sigma).
\end{split}
\]
This proves $\mathbf{R3e}$.  

\end{proof}

Now we show the compatibility of the residue maps with the co-restriction maps, proving axiom $\mathbf{R3b}$. 
\begin{lemma}\label{corestrictions-residues}
Let $\phi\colon K\xrightarrow{}L$ be a finite extension of fields and let $v$ a valuation on $K$ and let ${v}_i$ be all the extensions of $v$ to $L$. Let $\phi_i\colon k(v)\xrightarrow{}k(v_i)$ be the induced extensions on the residue fields. Then we have the following commutative diagram. \[\begin{tikzcd}
{H^1_{\mathrm{tame}}(K/k, \mathcal{G}_{i+1})} & {H^1(k(v),\mathcal{G}_{i})} \\
{H^1_{\mathrm{tame}}(L/k, \mathcal{G}_{i+1})} & {\bigoplus_{i}H^1(k(v_i),\mathcal{G}_{i}).}
\arrow["{\partial_v}", from=1-1, to=1-2]
\arrow["{\phi^*}", from=2-1, to=1-1]
\arrow["{(\partial_{v_i})}", from=2-1, to=2-2]
\arrow["{\sum_i \phi_i^*}", from=2-2, to=1-2]
\end{tikzcd}\]
\end{lemma}

\begin{proof}
  Consider the diagram
\[
\begin{xymatrix}{
		H^1_{\mathrm{tame}}(L/k,\mathcal{G}_{i+1}) \ar[r] \ar[d]_-{\phi^*} & \underset{i}{\bigoplus} ~ H^1_{\mathrm{tame}}(L_{v_i}/k,\mathcal{G}_{i+1}) \ar[r] \ar[d] & \underset{i}{\bigoplus} ~ H^1(k(v_i),\mathcal{G}_{i}) \ar[d]^-{\sum_{i}\phi_{v_i}^*} \\
		H^1_{\mathrm{tame}}(K/k,\mathcal{G}_{i+1}) \ar[r] & H^1_{\mathrm{tame}}(K_v/k,\mathcal{G}_{i+1}) \ar[r] & H^1(k(v), \mathcal{G}_{i})}
\end{xymatrix}
\]
in which the commutativity of the outer square is what we wish to prove.  The left square is commutative by $\mathbf{(R1)}$, since $L\otimes_{K}K_v\cong \bigoplus_{i} L_{v_i}$.  Thus, in order to prove the Lemma, we are reduced to proving it for the extensions $K_v \to L_{v_i}$.  Therefore, replacing $K$ by $K_v$ and $L$ by $L_{v_i}$, we may assume that $v$ is a complete valuation on $K$ that extends uniquely to $w$ on $L$. Let $\overline{\phi}$ denote the induced extension on residue fields.  By a standard argument (see \cite[Proof of Proposition 7.4.1]{Gille-Szamuely} for instance), we reduce to the case where $\phi$ and $\overline{\phi}$ are Galois extensions. We need to show the commutativity of the following diagram, in which all the vertical arrows are appropriate corestriction maps.
\[
\begin{xymatrix}{
		{H^1_{\mathrm{tame},w}(L,\mathcal{G}_{i+1})} \ar[r]^-{\cong} \ar[dd]_-{\phi^*}& {H^1(k(w),H^0(L^{sh},\mathcal{G}_{i+1}))} \ar[r] \ar[d]_-{{\phi^{sh}}^*}& {H^1(k(w),H^0(k(w)^{sep},\mathcal{G}_{i}))} \ar[d] \\
		& {H^1(k(w),H^0(K^{sh},\mathcal{G}_{i+1}))} \ar[r] \ar[d]_-{{\psi}^*} & {H^1(k(w),H^0(k(v)^{sep},\mathcal{G}_{i}))} \ar[d] \\
		H^1_{\mathrm{tame},v}(K,\mathcal{G}_{i+1}) \ar[r]^-{\cong} &  H^1(k(v),H^0(K^{sh}, \mathcal{G}_{i+1})) \ar[r] & H^1(k(v),H^0(k(w)^{sep},\mathcal{G}_{i})) }
\end{xymatrix}
\]
Since the analogue of the Lemma holds for $H^0(-, \mathcal{G}_{i+1})$, applying it to $L^{sh}\xrightarrow{\phi^{sh}}K^{sh}$ and then applying $H^1(k(w),-)$, we conclude that the top right square is commutative.  The bottom right square commutes because of the functoriality of corestriction. So, it suffices to show that the diagram on the left commutes.

We have the following commutative diagram of groups, in which $H$, $H'$ and $H''$ are defined so as to have exact rows and columns.
\[
\begin{xymatrix}{
		0 \ar[r] & H'' \ar[r]\ar[d] & H' \ar[r]\ar[d] & H  \ar[r]\ar[d] & 0 \\
		0 \ar[r] & \Gamma_L \ar[r]\ar[d] & \Gamma_K \ar[r]^-{p_1}\ar[d]^-{p_4} & \Gal(L/K) \ar[r]\ar[d]^-{p_2} &  0 \\
		0 \ar[r] & \Gamma_{k(w)} \ar[r] & \Gamma_{k(v)} \ar[r]^-{p_3} & \Gal(k(w)/k(v)) \ar[r]&  0}
\end{xymatrix}
\]
Choose (set-theoretic) splittings $s_1, s_2, s_3, s_4$ of $p_1, p_2, p_3, p_4$ respectively such that the lower right square in the above diagram commutes also when the maps $p_i$ are replaced by the splittings $s_i$.  The map ${\phi^{sh}}^*$ is given by applying the functor $H^1(k(w),-)$ to the norm map $N: H^0(L^{sh},\mathcal{G}_{i+1})\xrightarrow{}H^0(K^{sh},\mathcal{G}_{i+1})$.  Let $\alpha \in H^1_{\mathrm{tame}}(L/k,\mathcal{G}_{i+1})$. Then $\phi^*(\alpha)$ is represented by the $1$-cocycle that sends $\sigma \in \Gamma_K$ to 
\[
\begin{split}
	\phi^*(\alpha)(\sigma) &= \sum_{\tau\in\Gal(L/K)}s_1(\tau)^{-1}\alpha(\overline{s_1(\tau)\sigma s_1(\tau\sigma)^{-1}})\\
	&=\sum_{\omega\in\Gal(k(w)/k(v))}\sum_{h\in H}s_1(hs_2(\omega))^{-1}\alpha(\overline{s_1(hs_2(\omega))\sigma s_1(hs_2(\omega)\sigma)^{-1}})\\
	&= \sum_{\omega\in\Gal(k(w)/k(v))}s_3(\omega)^{-1}\sum_{h\in H}s_1(h)^{-1} \alpha(s_3(\omega)\overline{\sigma}s_3(\omega\overline{\sigma})^{-1}) \\
	&= \sum_{\omega\in\Gal(k(w)/k(v))}s_3(\omega)^{-1} {N}(\alpha(s_3(\omega)\overline{\sigma}s_3(\omega\overline{\sigma})^{-1}) \\
	&= \psi^* \circ {\phi^{sh}}^*(\alpha)(\sigma),
\end{split}
\]
where we have used the equalities $\overline{s_1(hs_2(\omega)}=\overline{s_1(\omega)}$, $\overline{s_1(hs_2(\omega)\sigma)^{-1}}=\overline{s_1(s_2(\omega)\sigma)^{-1}}$ and $s_1(s_2(\omega))=s_4(s_3(\omega))$ coming from the above diagram of Galois groups and the fact that the action of $s_4(s_3(\omega))$ on $H^0(L^{sh},\mathcal{G}_{i+1}))$ is the same as that of $s_3(\omega)$. This proves the lemma.
\end{proof}



Note that the data $\mathbf{(D1)}-\mathbf{(D4)}$ does not make $H^1_{\mathrm{tame}}(-/k, \mathcal{G}_{\ast})$ into a cycle pre-module since the residue map $\mathbf{(D4)}$ need not map $H^1_{\mathrm{tame}}(K/k, \mathcal{G}_{i+1})$ into $H^1_{\mathrm{tame}}(k(v)/k, \mathcal{G}_{i})$. However, by making a further assumption on $\mathcal{G}$, we can get a similar structure.

\begin{conventions}
\label{additional_assumptions_on_G}
From now on, we make the following additional assumption on $\mathcal{G}$: for a discrete valuation $v$ with residue field $k(v)$ and fraction field $K$ with $(k(v):k(v)^p)\le p^i$, we assume that $H^1_{\text{\'et}}(K^{sh}, \mathcal{G}_{i+1})=0$.
\end{conventions}

\begin{remark}
In the case of logarithmic de Rham-Witt sheaves $\@G = \nu_r(q)$, Conventions \ref{additional_assumptions_on_G} hold by a result of Kato \cite[page 150]{Kato1986}. 
\end{remark}

\begin{lemma}
\label{lemma tame=all}
For a sheaf $\@G$ satisfying Notation \ref{notation G} and Conventions \ref{additional_assumptions_on_G}, we have 
\[
H^1_{\mathrm{tame},v}(K, \mathcal{G}_{i+1}) = H^1(K, \mathcal{G}_{i+1}),
\]
for any valuation $v$ on $K$ such that $\deg (k(v):k(v)^p)\le p^i$.
\end{lemma}
\begin{proof}
Observe the Hochschild-Serre spectral sequence for the field extension $K_v\rightarrow K_v^{sh}$ given by $E_2^{a,b}=H^a_{\text{\'et}}(k(v),H^b_{\text{\'et}}(K_v^{sh},\mathcal{G}_{i+1})) \Rightarrow H^{a+b}_{\text{\'et}}(K_v,\mathcal{G}_{i+1})$. By the assumption on $\mathcal{G}$, it follows that $E_2^{0,1}=0$.  This implies that we have an isomorphism $H^1_{\text{\'et}}(K_v,\mathcal{G}_{i+1})\xrightarrow{\sim}H^1_{\text{\'et}}(k(v),H^0(K_v^{sh},\mathcal{G}_{i+1}))$.
\end{proof}

Therefore, Definition \ref{tameResidue} gives a residue map $\partial_v:M_{i+1}(K)\rightarrow{}M_i(k(v))$. This map, along with the data $\mathbf{(D1)},\mathbf{(D2)}$ and $\mathbf{(D3)}$ gives rise to what we call a \textit{weak cycle pre-module} structure on $\{{M_i}\}_{i\in \mathbb{Z}}$.

\begin{definition}\label{def_differential}
Let $\mathcal{G}$ be an $\mathbb{A}^1$-invariant \'etale sheaf satisfying Conventions \ref{additional_assumptions_on_G}. Let 
\[
C_i(X,\mathcal{G}_{q}) = \bigoplus_{x\in X_{(i)}} H^1_{\text{\'et}}(k(x),\mathcal{G}_{i+q}).
\]
We define $d:C_{i+1}(X,\mathcal{G}_{q})\rightarrow C_i(X,\mathcal{G}_{q})$ as follows.  Let $x\in X_{(i+1)}$, $y\in X_{(i)}$ and $\alpha\in H^1(k(x),\mathcal{G}_{i+1+q})$. If $y\notin Z=\overline{\{x\}}$, then we set the $y$- component of $d(\alpha)$ to be zero.  Suppose $y\in Z=\overline{\{x\}}$ and consider $Z$ with the reduced induced subscheme structure. Then $R=\mathcal{O}_{Z,y}$ is a $1$-dimensional $k$-algebra with residue field $k(y)$ and fraction field $k(x)$. Let $R'$ be its normalization; this is a $1$-dimensional semilocal finite $R$-algebra. For each valuation $w$ of $k(x)$ corresponding to the maximal ideals of $R'$, we get a finite extension $\varphi_w\colon k(y)\xrightarrow{}k(w)$.  We define the $y$-component of $d(\alpha)$ in this case to be $\sum _w \varphi^*_w\circ \partial_w(\alpha)$. 
\end{definition}

\begin{theorem}
\label{theorem existence_Kato_complex}
Let $\mathcal{G}$ be an $\mathbb{A}^1$-invariant \'etale sheaf satisfying Conventions \ref{additional_assumptions_on_G}. Then the groups $M_i(K)=H^1_{\text{\'et}}(K,\mathcal{G}_i)$ for $i\in \mathbb{Z}$ satisfy axioms {\bf (R1a-c)}, {\bf (R2a-c)} and {\bf (R3a-e)}. 
Consequently, for $q\ge 0$, the groups $C_i(X,\mathcal{G}_{q})$ along with the differentials defined in Definition \ref{def_differential} form a Kato Complex
\[
C(X,\mathcal{G}_q): 0\rightarrow\bigoplus_{x\in X_{(d)}} H^1_{\text{\'et}}(k(x),\mathcal{G}_{d+q})\xrightarrow{d} \cdots \xrightarrow{d}  \bigoplus_{x\in X_{(0)}}H^1_{\text{\'et}}(k(x),\mathcal{G}_q)\rightarrow 0.
\] 
\end{theorem}
\begin{proof}
This follows immediately from Lemmas \ref{lemma tame=all}, \ref{lemma res}, \ref{lemma cores}, \ref{lemma K-action} and Proposition \ref{prop R3}.  The existence of the Kato complex is exactly analogous to the existence of the Rost cycle complex.
\end{proof}

\begin{remark}
It can be verified that the functor $M_*(-)$ in Theorem \ref{theorem existence_Kato_complex} above satisfies the cycle module axioms from \cite{Rost1996}. 
\end{remark}

\section{Unramified cohomology}
\label{section unramified}

In this section we define the \textit{unramified cohomology} and the \textit{naive tame unramified cohomology} groups and prove several results about them generalizing those in \cite[Section 4]{Otabe2023}. Let $\mathcal G$ be as in Notation \ref{notation G} and Convention \ref{additional_assumptions_on_G}.

\begin{definition}
For a field $K$, we define the unramified cohomology as \[H^1_\mathrm{{ur}}(K/k, \mathcal{G}) := \bigcap_{v\in Val(K/k)}\Img(H^1(\mathcal{O}_v,\mathcal{G}) \rightarrow H^1(K, \mathcal{G})).\] 
For a normal variety $X$ over $k$, we define the unramified cohomology as \[H^1_\mathrm{{ur}}(X, \mathcal{G}) := \bigcap_{x\in X^{(1)}}\Img(H^1(\mathcal{O}_{X,x},\mathcal{G}) \rightarrow H^1(k(X), \mathcal{G})).\] We define the unramified tame cohomology of $X$ as \[H^1_{\mathrm{tame, ur}}(X/k, \mathcal{G}) := H^1_{\mathrm{ur}}(X, \mathcal{G}) \cap H^1_{\mathrm{tame}}(k(X)/k, \mathcal{G}).\]
\end{definition}

For the unramified cohomology to be functorial on $Sm_k$, we need to make an additional assumption on $\mathcal{G}$. The following definition is taken from \cite[Definition 2.1.4]{Colliot-Thelene}.

\begin{definition}\label{specialization} We say that $\mathcal{G}$ satisfies the \textit{specialization} property for a regular local ring $A$ if $\Ker(H^1(A,\mathcal{G})\rightarrow H^1(k(A),\mathcal{G})) \subset \Ker(H^1(A,\mathcal{G})\rightarrow H^1(k_A,\mathcal{G})),$ where $k(A)$ and $k_A$ are the fraction field and residue field of $A$ respectively.
\end{definition}

The following result is a special case of a theorem of Colliot-Thel\`ene(\cite[Theorem 2.1.10]{Colliot-Thelene}).

\begin{theorem}\label{functoriality_unramified}
Suppose $\mathcal{G}$ satisfies the specialization property for every discrete valuation ring. Then $H^1_{\mathrm{ur}}(-, \mathcal{G})$ defines a contravariant functor on the category of regular integral schemes over $k$.
\end{theorem}

In \cite[Theorem 4.3]{Totaro}, Totaro shows the existence of an exact sequence\[
0\rightarrow H^1_{\mathrm{ur},v}(K/k, \Omega^{i+1}_{\mathrm{log},K})\rightarrow H^1_{\mathrm{tame},v}(K,\Omega^{i+1}_{\mathrm{log},K}) \xrightarrow{\partial_v} H^1(k(v),\Omega^{i}_{\mathrm{log},k(v)})\rightarrow 0
\] and uses it to show that the pushforward map on the tame cohomology groups induce a pushforward the unramified tame cohomology groups in the case of finite surjective morphisms. We generalize this to a certain class of $\mathbb{A}^1$-invariant sheaves which we call \textit{good} sheaves.

\begin{definition}
Let $\mathcal{G}$ be an $\mathbb{A}^1$-invariant \'etale sheaf with transfers. We have the cup product in Galois cohomology $H^1(K,\mathbb{Z}/pZ)\times H^0(K,\mathcal{G}) \xrightarrow{\cup} H^1(K,\mathcal{G}).$ We have $H^1(K,\mathbb{Z}/pZ)\cong K/\mathcal{P}(K)$, where $\mathcal{P}$ is the Artin-Schreier map. For $a\in K$ and $s \in H^0(K,\mathcal{G})$, we define the symbol $[a,s\}:=[a]\cup s$. We say that $\mathcal{G}$ is \textit{good} if it satisfies the following conditions:
\begin{enumerate}
\item\label{conv1}  $H^1(K, \mathcal{G})$ is generated by symbols of the form $[a,s\}$, where $a\in K$ and $s \in H^0(K,\mathcal{G})$.

\item \label{conv2} $H^1_{\mathrm{tame},v}(K, \mathcal{G})$ is generated by symbols of the form $[a,s\}$, where $a\in \mathcal{O}_v$ and $s \in H^0(K,\mathcal{G})$.

\item\label{conv3} For any discrete valuation ring $R$ with valuation $v$, the restriction map $\mathcal{G}_{-1}(R) \xrightarrow{} \mathcal{G}_{-1}(k(v))$  is surjective. 

\item\label{conv4} The following sequence is exact. \[
0\rightarrow \mathcal{G}(\mathcal{O}_{v}) \rightarrow \mathcal{G}(K) \xrightarrow{\partial_v} \mathcal{G}_{-1}(k(v)) \rightarrow 0.
\]
\end{enumerate}
\end{definition}

The following Lemma enables us to generalize the proof of Totaro to the above defined class of sheaves.

\begin{lemma}
Let $t$ be any uniformizer of $v$ and let $\mathcal{G}$ be a good sheaf. Then $\mathcal{G}(K)$ is generated by \[ \Img(\mathcal{G}_{-1}(\mathcal{O}_{v}) \xrightarrow{\{t\}\cdot} \mathcal{G}(\mathcal{O}_{v}) \rightarrow \mathcal{G}(K)) \] and \[ \Img(\mathcal{G}(\mathcal{O}_{v}) \rightarrow \mathcal{G}(K)).\]
\end{lemma}

\begin{proof}
Let $s$ be a set theoretic section of the surjective restriction map $\mathcal{G}_{-1}(\mathcal{O}_v) \xrightarrow{} \mathcal{G}_{-1}(k(v))$. Let $f$ be the composite map $\mathcal{G}_{-1}(k(v)) \xrightarrow{s} \mathcal{G}_{-1}(\mathcal{O}_v) \xrightarrow{\{t\}\cdot} \mathcal{G}(\mathcal{O}_v) \rightarrow \mathcal{G}(K)$, where the last map is the restriction map.

Since the composite $\mathcal{G}_{-1}(\mathcal{O}_v) \xrightarrow{\{t\}\cdot} \mathcal{G}(\mathcal{O}_v) \rightarrow \mathcal{G}(K) \xrightarrow{\partial_v} \mathcal{G}_{-1}(k(v))$ coincides with the restriction map $\mathcal{G}_{-1}(\mathcal{O}_v) \rightarrow \mathcal{G}_{-1}(k(v))$, we have $\partial_v \circ f = id$. Now, let $\alpha \in \mathcal{G}(K)$; then \[\partial_v(\alpha - f(\partial_v(\alpha))) = \partial_v(\alpha) - \partial_v(f(\partial_v(\alpha))) = \partial_v(\alpha) - \partial_v(\alpha) =0.\] Therefore, $\alpha - f(\partial_v(\alpha)) \in \Img(\mathcal{G}(\mathcal{O}_{v}) \rightarrow \mathcal{G}(K))$. Also, by definition, \[f(\partial_v(\alpha))) \in \Img(\mathcal{G}_{-1}(\mathcal{O}_{v}) \xrightarrow{\{t\}\cdot} \mathcal{G}(\mathcal{O}_{v}) \rightarrow \mathcal{G}(K)).\] Hence, we are done.
\end{proof}

With these assumptions, we can generalize the exact sequence of \cite[Theorem 4.3]{Totaro} to $\mathcal{G}$.

\begin{proposition}\label{exact_sequence}
Let $\mathcal{G}$ be a good sheaf and let $K$ be a discrete valuation ring with valuation $v$. We have an exact sequence \[
0\rightarrow H^1_{\mathrm{ur},v}(K/k, \mathcal{G})\rightarrow H^1_{\mathrm{tame},v}(K,\mathcal{G}) \xrightarrow{\partial_v} H^1(k(v),\mathcal{G}_{-1})\rightarrow 0.
\]
\end{proposition}

\begin{proof}
Let $t$ be a uniformizer of $v$.

Let $[a,s\}\in H^1_{\mathrm{ur},v}(K, \mathcal{G})$. Then $s\in H^1_{\mathrm{ur},v}(K,G)$, which implies that $\partial_v(s)=0$. We have $\partial_v([a,s\})=[\overline{a},\partial_v(s)\} = [\overline{a}, 0\}$.

Now we define a map \[\varphi:H^1(k(v),\mathcal{G}_{-1}) \rightarrow H^1_{\mathrm{tame},v}(K,\mathcal{G})/H^1_{\mathrm{ur},v}(K, \mathcal{G})\] by sending $[\overline{a},s'\}$ to $[a,\{t\}\cdot s\}$, where $a$ is a lift for $\overline{a}$ and $s\in H^0(K,\mathcal{G})$ is such that $\partial_v(\{t\} \cdot s)=s'$ and $s\in H^1_{\mathrm{ur},v}(K,\mathcal{G})$. Now we show that $\varphi$ is well defined. Let $s_1, s_2 \in H^1_{\mathrm{ur},v}(K, \mathcal{G})$ be such that $\partial_v(\{t\} \cdot s_1) = \partial_v(\{t\} \cdot s_2)$ and let $a_1, a_2 \in \mathcal{O}_{v}$ be lifts of $\overline{a}$.  Then we have $\partial_v(\{t\}\cdot s_1 - \{t\}\cdot s_2) = 0$ so that $\{t\}\cdot s_1 - \{t\}\cdot s_2$ comes from $\mathcal{G}(\mathcal{O}_{v})$ by assumption \ref{conv4}. Therefore, $[a,\{t\}\cdot s_1\}-[a,\{t\}\cdot s_2\} \in H^1_{\mathrm{ur}}(K, \mathcal{G})$, which shows that $\varphi([\overline{a}, s'\})$ is independent of the choice of $s$ and $a$.

Clearly, we have $\partial_v\circ\varphi=id$.

Note that showing that the sequence is exact is the same as showing that $\varphi$ is an isomorphism. Since $\varphi$ is a section, it remains for us to show that it is surjective. Let $[a,s\}\in H^1_{\mathrm{tame},v}(K,\mathcal{G})$ be such that $a\in \mathcal{O}_v$ and $s\in H^0(K,\mathcal{G})$. By our assumption \ref{conv2}, $H^1_{\mathrm{tame},v}(K,\mathcal{G})$ is generated by such elements. By our assumption \ref{conv3}, $s$ is the linear combination of elements of the form $\{t\}\cdot s'$ and $s''$, where $s'$ and $s''$ come from $\mathcal{G}(\mathcal{O}_{v})$ and $\mathcal{G}_{-1}(\mathcal{O}_{v})$ respectively.

Therefore, $[a, s\}$ is the linear combination of elements of the form $[a, \{t\}\cdot s'\}$ and $[a,s''\}$. The former lie in the image of $\varphi$, while the latter lie in $H^1_{\mathrm{ur},v}(K, \mathcal{G})$. Therefore, $\varphi$ is surjective.

\end{proof}

Using this result, we can show that that the tame unramified cohomology groups admit pushforwards under finite surjective morphisms.

\begin{lemma}
Let $f:Y\rightarrow X$ be a finite surjective morphism of normal varieties over $k$. Then the co-restriction map $H^1_{\mathrm{tame}}(k(Y)/k, \mathcal{G}) \rightarrow H^1_{\mathrm{tame}}(k(X)/k, \mathcal{G})$ induces a map \[f_*: H^1_{\mathrm{tame, ur}}(Y/k, \mathcal{G})\rightarrow H^1_{\mathrm{tame, ur}}(X/k, \mathcal{G}).\]
\end{lemma}

\begin{proof}
We need to show that the co-restriction maps defined on the function fields sends an element of $H^1_{\mathrm{tame, ur}}(Y/k, \mathcal{G})$ to an element of $H^1_{\mathrm{tame, ur}}(k(X)/k, \mathcal{G}).$ By Proposition \ref{exact_sequence} and Lemma \ref{corestrictions-residues}, we have a commutative diagram with exact rows as follows.
\[\begin{tikzcd}
	0 & {H^1_{\mathrm{tame,ur}}(Y,\mathcal{G})} & {H^1_{\mathrm{tame}}(k(Y)/k,\mathcal{G})} & {\bigoplus_{y\in Y^{(1)}}H^1(k(y),\mathcal{G}_{-1})} \\
	0 & {H^1_{\mathrm{tame,ur}}(X,\mathcal{G})} & {H^1_{\mathrm{tame}}(k(Y)/k,\mathcal{G})} & {\bigoplus_{x\in X^{(1)}}H^1(k(x),\mathcal{G}_{-1})}
	\arrow[from=1-1, to=1-2]
	\arrow[from=1-2, to=1-3]
	\arrow[from=1-3, to=1-4]
	\arrow[from=1-3, to=2-3]
	\arrow[from=1-4, to=2-4]
	\arrow[from=2-1, to=2-2]
	\arrow[from=2-2, to=2-3]
	\arrow[from=2-3, to=2-4]
\end{tikzcd}\]
From the diagram, it is clear that any element of of $H^1_{\mathrm{tame, ur}}(Y/k, \mathcal{G})$ maps to an element of $H^1_{\mathrm{tame, ur}}(k(X)/k, \mathcal{G}).$
\end{proof}

\section{Unramified curve-tame cohomology}
\label{section unramified ct}

The unramified curve-tame cohomology groups were defined by Otabe in \cite{Otabe2023} for the sheaf of logarithmic K\"ahler differentials. We extend this notion to a more general class of $\mathbb{A}^1$-invariant sheaves and prove results analogous to those in \cite{Otabe2023}. All sheaves $\mathcal{G}$ considered in this section will be $\mathbb{A}^1$-invariant \'etale sheaves with transfers over $k$ satisfying the specialization property of Definition \ref{specialization} for every discrete valuation ring over $k$.

\begin{conventions} We use the following conventions for the rest of the article.\label{conventions_homology_section}
\begin{enumerate}
\item All fields considered will be finitely generated over a field $k$ of characteristic $p>0$.

\item Let $\mathcal G$ be as in Notation \ref{notation G} and Convention \ref{additional_assumptions_on_G}.

\item For a field extension $\varphi:K \rightarrow L$, we denote by $r_{L/K}$ the restriction map $\varphi_*: H^1(K, \mathcal{G}) \rightarrow H^1(L, \mathcal{G})$. If $\varphi$ is finite, we denote by $c_{L/K}$ the co-restriction map $\varphi^*: H^1(L, \mathcal{G}) \rightarrow H^1(K, \mathcal{G})$. Abusing notation, we shall use the same notation for the restrictions of these maps to the tame subgroup and the unramified subgroup.

\item $\mathcal G$ satisfies the specialization property of Definition \ref{specialization}. This implies that the unramified cohomology is a contravariant functor by Theorem \ref{functoriality_unramified}. 

\end{enumerate}
\end{conventions}

\begin{definition}
We define the unramified curve-tame cohomology group $H^1_{\mathrm{ct, ur}}(X, \mathcal{G})$ to be the subgroup of $H^1_{\mathrm{ur}}(X, \mathcal{G})$ consisting of elements $\alpha$ such that for any field $K$ and any morphism $C\rightarrow X$ from any normal $K$-curve $C$, the pullback $\alpha|C$ lies in $H^1_{\mathrm{tame, ur}}(C/K, \mathcal{G})$. 
\end{definition}

\begin{lemma}\label{functoriality_curve_tame}
Let $X\in Sm_k$ and let $Y$ be a smooth scheme over a field $K$ and let $f:Y\rightarrow X$ be a $k$-morphism. Then the restriction map $f^*:H^1_{\mathrm{ur}}(X,\mathcal{G}) \rightarrow H^1_{\mathrm{ur}}(Y.\mathcal{G})$ induces a map $H^1_{\mathrm{ct, ur}}(X/k, \mathcal{G}) \rightarrow H^1_{\mathrm{ct, ur}}(Y/K, \mathcal{G})$.
Moreover, if $K \rightarrow L$ is a field extension and $Z$ is a smooth scheme over $Y$ and $Z \rightarrow Y$ is a $K$-morphism, the restriction maps satisfy the obvious functoriality properties.

\end{lemma}

\begin{proof}
Suppose $\alpha\in H^1_{\mathrm{ct, ur}}(X/k, \mathcal{G})$.  Let $L$ be a finitely generated extension of $K$ and let $C\rightarrow Y$ be a $K$-morphism from a normal $L$-curve $C$. Then $(f^*\alpha)|C$ is the restriction of $\alpha$ along the composite morphism $C \rightarrow Y \rightarrow X$. Since $\alpha \in H^1_{\mathrm{ct, ur}}(X/k,\mathcal{G}),$ this lies in $H^1_{\mathrm{ct, ur}}(C/L, \mathcal{G})$. Therefore, $f^*\alpha \in H^1_{\mathrm{ct, ur}}(Y/K, \mathcal{G})$. The functoriality follows from the corresponding result for $H^1_{\mathrm{ur}}(-,\mathcal{G})$.
\end{proof}

We note down an extension of \cite[Proposition 4.12, (1),(2)]{Otabe2023} below.

\begin{proposition}
Let $C$ be a normal $k$-curve, then $H^1_{\mathrm{tame, ur}}(C/k, \mathcal{G}) = H^1_{\mathrm{ct, ur}}(C/k, \mathcal{G})$.
\end{proposition}

\begin{proof}
By considering the identity map $C \rightarrow C$, it is clear from the definition that \[H^1_{\mathrm{ct, ur}}(C/k, \mathcal{G}) \subset H^1_{\mathrm{tame, ur}}(C, \mathcal{G}).\] For the other inclusion, let $\alpha \in H^1_{\mathrm{tame, ur}}(C/k, \mathcal{G})$. Let $K$ be a finitely generated field over $k$ and $f:D \rightarrow C$ be a $k$-morphism from a normal $K$-curve. It suffices for us to show that $f^*(\alpha) \in H^1_{\mathrm{tame, ur}}(D/K, \mathcal{G})$. Let $E = \overline{f(D)}$; then $E=C$ or $E = \{x\}$ for some $x\in C_{(0)}$. In the first case, $f$ is dominant, and therefore, $f^*(\alpha) \in H^1_{\mathrm{tame, ur}}(D/K, \mathcal{G}).$

Now suppose $E= \{x\}$. It suffices to show that the map $H^1(k(x),\mathcal{G}) \rightarrow H^1_{\mathrm{ur}}(D,\mathcal{G})$ factors through $H^1_{\mathrm{tame, ur}}(D/K, \mathcal{G})$. Since $k(x)/k$ is finite, for any valuation $v \in Val(K(D)/K)$, we have that $v_{|k(x)}$ is trivial. Therefore, the image of $H^1(k(x),\mathcal{G}) \rightarrow H^1_{\mathrm{ur}}(D,\mathcal{G})$ lies in $H^1_{\mathrm{ur}}(K(D)/K,\mathcal{G}) \subset H^1_{\mathrm{tame, ur}}(D/K, \mathcal{G})$.
\end{proof}

\begin{proposition}\label{proper_equality}
Let X be a proper smooth geometrically connected variety over $k$, then we have \[H^1_{\mathrm{ct, ur}}(X/k,\mathcal{G}) = H^1_{\mathrm{ur}}(X, \mathcal{G}) = H^1_{\mathrm{ur}}(k(X)/k, \mathcal{G}).\]
\end{proposition}

\begin{proof}

The equality $H^1_{\mathrm{ur}}(X, \mathcal{G}) = H^1_{\mathrm{ur}}(k(X)/k, \mathcal{G})$ is a special case of \cite[Proposition 2.18 (e)]{Colliot-Thelene}.

For the other equality, let $\alpha \in H^1_{\mathrm{ur}}(X, \mathcal{G})$. Let $K$ be a finitely generated extension over $k$ and let $f: C \rightarrow X$ be a morphism from a normal $k$-curve $C$. To prove the proposition, it suffices to show that $f^*(\alpha) \in H^1_{\mathrm{tame, ur}}(C/K,\mathcal{G})$. The morphism $f$ factors through the proper morphism $X_K \rightarrow X$; therefore, the $K$-morphism can be extended to some normal compactification $\overline{C}$ of $C$. Therefore, $f^*(\alpha)\in H^1_{\mathrm{ur}}(\overline{C}/K,\mathcal{G}) \subset H^1_{\mathrm{tame, ur}}(C/K,\mathcal{G})$.
\end{proof}

\section{Pairing with Suslin homology}
\label{section pairing}

In \cite[Section 5]{Otabe2023}, Otabe defines a pairing of the unramified curve-tame cohomology with the Suslin homology. We extend this to a more general class of $\mathbb{A}^1$-invariant sheaves with transfers. First, we recall the definition of Suslin homology. Throughout this section, $\mathcal{G}$ will denote an $\mathbb{A}^1$-invariant \'{e}tale sheaf with transfers over a field $k$ of characteristic $p>0$ satisfying the specialization property \ref{specialization}.

\begin{definition}\label{Suslin_homology}
For a scheme $X \in Sm_k$, we define the Suslin homology as \[H^S_0(X) := \Coker(\Cor_k(\mathbb{A}^1_k, X) \xrightarrow{s_0^* - s_1^*} Z_0(X)),\] where $Z_0(X)$ is the group of zero cycles and $s_0$ and $s_1$ are the morphisms $\Spec k \rightarrow \mathbb{A}^1$ corresponding to the points $0$ and $1$ respectively.
\end{definition}

We are now ready to define the pairing.

\begin{definition}\label{pairing_on_cycles}
Let $X$ be a smooth geometrically connected variety over $k$. We define a pairing $\langle-,-\rangle: Z_0(X) \times H^1_{\mathrm{ct, ur}}(X/k, \mathcal{G}) \rightarrow H^1(k, \mathcal{G})$ as the restriction of the composite \[
Z_0(X) \times  H^1_{\mathrm{ur}}(X,\mathcal{G}) \xrightarrow{(x^*)_{x\in X_{(0)}}} \bigoplus_{x\in X_{(0)}} H^1(k(x),\mathcal{G})\xrightarrow{\sum_{x\in X_{(0)}} c_{k(x)/k}} H^1(k, \mathcal{G}),
\] where $x^*$ is the restriction map corresponding to the morphism $\Spec k(x) \rightarrow X$.
\end{definition}

To show that this pairing induces a pairing $H^S_0(X) \times H^1_{\mathrm{ct, ur}}(X/k, \mathcal{G}) \rightarrow H^1(k, \mathcal{G})$, we need some preliminary lemmas; the proofs are completely analogous to the corresponding results in \cite{Otabe2023}.

\begin{lemma}\label{well_definedness_aux_1}
With the same notation as in Definition \ref{pairing_on_cycles}, let $\phi: C \rightarrow X$ be a $k$-morphism from a normal $k$-curve $C$ to $X$. Then for any $z\in Z_0(C)$ and $\alpha \in H^1_{\mathrm{ct, ur}}(X/k, \mathcal{G})$, we have $\langle\phi_*(z),\alpha\rangle = \langle z,\phi^*(\alpha)\rangle$.
\end{lemma}

\begin{proof}
We may assume without loss of generality that $z$ is a single point. Let $x=\phi(z)$; then we have $\phi_*(z) = (k(z):k(x))x \in Z_0(X)$. Therefore, we have \[
\langle\phi_*(z),\alpha\rangle = c_{k(x)/k}((\phi_*z)^*\alpha) = (k(z):k(x))\cdot c_{k(x)/k}(x^*\alpha).
\]
However, since $(k(z):k(x)) = c_{k(z)/k(x)}\circ r_{k(z)/k(x)}$, we get 
\[
\begin{split}
\langle\phi_*(z),\alpha\rangle = c_{k(x)/k}(c_{k(z)/k(x)}(r_{k(z)/k(x)}(x^*\alpha))) & = c_{k(z)/k}(r_{k(z)/k(x)}(x^*\alpha)) \\
& = c_{k(z)/k}(z^*\phi^*\alpha) = \langle z,\phi^*\alpha \rangle. 
\end{split}
\] 
Here, the second equality follows from axiom \textbf{R1b} and the third one from the functoriality of the pullback.
\end{proof}

\begin{lemma}\label{well_definedness_aux_2}
Let $\Gamma \rightarrow C$ be a finite surjective morphism between two normal $k$-curves. Let $x: \Spec k(x) \rightarrow C$ be a closed point of $C$ and let $\Gamma_x = \phi^{-1}(x)$ be the inverse image of $x$. Then for any $\alpha \in H^1_{\mathrm{ct, ur}}(\Gamma/k, \mathcal{G})$, we have $\langle \Gamma_x, \alpha\rangle = \langle x, \phi_*\alpha \rangle$.
\end{lemma}

\begin{proof}
Let $\pi$ be a uniformizer of the discrete valuation ring $\mathcal{O}_{C,x}$; for each $y\in \Gamma$ mapping to $x$, let $\pi_y$ be a uniformizer for $\mathcal{O}_{\Gamma,y}$. Let $e_y = v_y(\phi^*\pi_x)$ be the ramification index of the map $\mathcal{O}_{C,x} \rightarrow \mathcal{O}_{\Gamma,y}$. By, the axioms \textbf{R2b} and Lemma \ref{corestrictions-residues}, we have \begin{align*}
\langle x, \phi_*(\alpha)\rangle &= c_{k(x)/k}(x^*\phi_*(\alpha))\\
&=c_{k(x)/k}(\partial_x(\phi^*(\alpha)\cdot \pi)\\
&=c_{k(x)/k}(\partial_x(c_{k(\Gamma)/k(C)}(\alpha\cdot \phi^*\pi)))\\
&=c_{k(x)/k}(\sum_{y} c_{k(y)/k(x)} (\partial_y(\alpha\cdot\phi^*\pi)))\\
&=\sum_y e_y c_{k(y)/k}(\partial_y(\alpha\cdot \pi_y))\\
&=\sum_y e_y c_{k(y)/k}(y^*\alpha)\\
&=\sum_y e_y \langle y,\alpha\rangle = \langle\Gamma_x, \alpha \rangle.	
\end{align*}
\end{proof}

\begin{proposition}\label{well_definedness_of_pairing}
Let $X$ be a smooth geometrically connected variety over $k$. Let $\mathcal{G}$ be such that the projection map $\mathbb{A}^1_k \rightarrow \Spec k$ induces an isomorphism $H^1(k,\mathcal{G}) \xrightarrow{\sim} H^1_{\mathrm{tame, ur}}(\mathbb{A}^1_k/k, \mathcal{G})$.  The pairing $\langle-,-\rangle$ defined in Definition \ref{pairing_on_cycles} factors through the Suslin homology $H^S_0(X)$ to form a pairing\[
\langle-,-\rangle: H^S_0(X) \times H^1_{\mathrm{ct, ur}}(X/k, \mathcal{G}) \rightarrow H^1(k, \mathcal{G}).
\]
\end{proposition}

\begin{proof}
Let $\alpha \in H^1_{\mathrm{ct, ur}}(X/k, \mathcal{G})$ and let $\Gamma \in \Cor_k(\mathbb{A}^1_k, X)$ be an elementary correspondence. Let $\phi: \Gamma \rightarrow \mathbb{A}^1_k = \Spec k[t]$ and $\psi: \Gamma \rightarrow X$ be the projections to $\mathbb{A}_k^1$ and $X$ respectively. By definition, $\phi$ is finite and surjective. In the notation of Definition \ref{Suslin_homology}, $(s_0^*-s_1^*)(\Gamma) = \psi_*(\phi ^{-1}(0) - \phi ^{-1}(1))$. 

Therefore, to show that the pairing $\langle-,-\rangle$ factors through $H^S_0(X)$, it suffices to show that $\langle\psi_*(\phi ^{-1}(0) - \phi ^{-1}(1)), \alpha\rangle = 0$. Note that we have
\[
\begin{split}
\phi ^{-1}(0) - \phi ^{-1}(1) = \phi^*\Div_{\mathbb{A}^1_k}(t/t-1) = \Div_{\Gamma}(\phi^*t/\phi^*t-1) & = \pi_*\Div_{\widetilde{\Gamma}}(\phi^*t/\phi^*t-1) \\ 
& = \pi_*(\widetilde{\phi ^{-1}(0)} - \widetilde{\phi ^{-1}(1)}),
\end{split}
\] 
where $\pi: \widetilde{\Gamma} \rightarrow \Gamma$ is the normalisation and $\widetilde{\phi ^{-1}(i)}$ is the pullback of $\phi ^{-1}(i)$ along $\pi$. Therefore, by Lemma \ref{well_definedness_aux_1}, we have \[
\langle\psi_*(\phi ^{-1}(0) - \phi ^{-1}(1)), \alpha\rangle = \langle\psi_*(\pi_*(\widetilde{\phi ^{-1}(0)} - \widetilde{\phi ^{-1}(1)})), \alpha\rangle = \langle \widetilde{\phi ^{-1}(0)} - \widetilde{\phi ^{-1}(1)}, (\psi\circ\pi)^*\alpha\rangle.
\]

So, it suffices to show that $\langle \widetilde{\phi ^{-1}(0)}, (\psi\circ\pi)^*\alpha\rangle = \langle \widetilde{\phi ^{-1}(1)}, (\psi\circ\pi)^*\alpha\rangle.$ By Lemma \ref{well_definedness_aux_2} , we have $\langle \widetilde{\phi ^{-1}(i)}, (\psi\circ\pi)^*\alpha\rangle = s_i^*(\phi\circ\pi)_*(\alpha)$.

Each $s_i$ is a section of the projection morphism $\mathbb{A}^1_k \rightarrow \Spec k$, therefore, by out assumption on $\mathcal{G}$ that $H^1(k,\mathcal{G}) \xrightarrow{\sim} H^1_{\mathrm{tame, ur}}(\mathbb{A}^1_k/k, \mathcal{G})$ is an isomorphism, we get that $s_0^* = s_1^*$ on $H^1_{\mathrm{tame, ur}}(\mathbb{A}^1_k/k,\mathcal{G})$, thus proving the proposition.

\end{proof}

We prove the following compatibility properties for the pairing $\langle-,-\rangle$, which will be used to prove Proposition \ref{curve_tame_equals_cohomology_of_base_field}.

\begin{lemma}\label{curve_tame_generic_compatibility}
Let $K = k(X)$ be the function field of $X$ and let $\eta\in (X_K)_{(0)}$ be the point corresponding to the generic point of $X$. Let $\mathcal{G}$ be as in Proposition \ref{well_definedness_of_pairing}. Then the composite map\[
H^1_{\mathrm{ct, ur}}(X/k, \mathcal{G}) \rightarrow H^1_{\mathrm{ct, ur}}(X_K/K, \mathcal{G}) \xrightarrow{\langle \eta,-\rangle} H^1(K,\mathcal{G})
\]
coincides with the inclusion $H^1_{\mathrm{ct, ur}}(X/k, \mathcal{G}) \subset H^1(K,\mathcal{G})$.
\end{lemma}

\begin{proof}
Let $\pi: X_K \rightarrow X$ be the projection. Consider the following diagram. Clearly, the lemma follows from the commutativity of the diagram.\[\begin{tikzcd}
	{H^1_{\mathrm{ct, ur}}(X/k, \mathcal{G})} & {H^1_{\mathrm{ct, ur}}(X_K/K, \mathcal{G})} \\
	{H^1_{\mathrm{ur}}(X, \mathcal{G})} & {H^1_{\mathrm{ur}}(X_K, \mathcal{G})} \\
	& {H^1(K, \mathcal{G})}
	\arrow["{\pi^*}", from=1-1, to=1-2]
	\arrow[hook, from=1-1, to=2-1]
	\arrow[hook, from=1-2, to=2-2]
	\arrow["{\pi^*}", from=2-1, to=2-2]
	\arrow[hook, from=2-1, to=3-2]
	\arrow["{\eta^*}", from=2-2, to=3-2]
\end{tikzcd}\]

The upper square is commutative by Lemma \ref{functoriality_curve_tame}. By applying the same lemma to the morphisms $\Spec K \xrightarrow{\eta} \Spec X_K \rightarrow X$, we get the commutativity of the bottom triangle.
\end{proof}

\begin{lemma}\label{curve_tame_special_compatibility}
Let $x \in X_{(0)}$ be such that $k(x)$ is separable over $k$. Let $\mathcal{G}$ be as in Proposition \ref{well_definedness_of_pairing}. Let $\alpha \in H^1_{\mathrm{ct, ur}}(X/k, \mathcal{G})$. Then for any finitely generated field extension $K/k$, we have $\langle x,\alpha\rangle_K = \langle x_K, \alpha_K\rangle$, where $x_K \in Z_0(X_K)$ and $\alpha_K \in H^1_{\mathrm{ct, ur}}(X_K/K, \mathcal{G})$ are the pullbacks of $x$ and $\alpha$ respectively along the projection morphism $\pi: X_K\rightarrow X$. 
\end{lemma}

\begin{proof}
Since $k(x)/k$ is a separable extension, the fibre $\pi^{-1}(x) = \Spec(K\otimes_k k(x))$ is a product of fields. Therefore, $x_K = \sum_{\pi(y)=x} y$ in $Z_0(X_K)$.We have \begin{align*}\langle x, \alpha \rangle_K &=  r_{K/k}(c_{k(x)/k}(x^*\alpha))\\
&=\sum_{\pi(y)=x} c_{K(y)/K}(r_{K(y)/K(x)}(x^*\alpha))\\
&=\sum_{\pi(y)=x} c_{K(y)/K}(y^*\alpha_K)\\
&=\langle x_K,\alpha_K\rangle,
\end{align*} where the second equality comes from axiom \textbf{R1c}.
\end{proof}

\begin{theorem}\label{curve_tame_equals_cohomology_of_base_field}
Let $X$ be a smooth geometrically connected variety over $k$. Let $\mathcal{G}$ be as in Proposition \ref{well_definedness_of_pairing}. Suppose that for any finitely generated field $K$ over $k$, the degree map $H^S_0(X_K) \rightarrow \mathbb{Z}$ is an isomorphism. Then the pullback $H^1(k, \mathcal{G}) \rightarrow H^1_{\mathrm{ct, ur}}(X/k,\mathcal{G}) $ is an isomorphism.
\end{theorem}

\begin{proof}
Since the degree map from Suslin homology is surjective, there exists an element $z\in Z_0(X)$ of degree $1$. To show that $H^1(k, \mathcal{G}) \rightarrow H^1_{\mathrm{ct, ur}}(X/k,\mathcal{G}) $ is injective, it suffices to show that the composition \[
H^1(k, \mathcal{G}) \rightarrow H^1_{\mathrm{ct, ur}}(X/k, \mathcal{G}) \xrightarrow{\langle z,-\rangle} H^1(k, \mathcal{G})
\] is the identity.
We have the following commutative diagram: \[\begin{tikzcd}
	{H^1(k,\mathcal{G)}} & {H^1_{\mathrm{ct, ur}}(k,\mathcal{G)}} \\
	{H^1(k,\mathcal{G)}} & {H^1(k(z),\mathcal{G)}}
	\arrow[from=1-1, to=1-2]
	\arrow[from=1-1, to=2-1]
	\arrow["{r_{k(z)/k}}", from=1-1, to=2-2]
	\arrow["{z^*}", from=1-2, to=2-2]
	\arrow["{c_{k(z)/k}}"', from=2-2, to=2-1]
\end{tikzcd}.\] Since $c_{k(z)/k} \circ r_{k(z)/k} = \deg(z) =1$, we are done.

We now prove the surjectivity of the map. Let $\alpha \in H^1_{\mathrm{ct, ur}}(X/k, \mathcal{G})$. Let $K=k(X)$ be the fraction field of $X$ and let $\eta \in (X_K)_{(0)}$ be the point corresponding to the generic point of $X$. Since the degree map $H^S_0(X)\rightarrow \mathbb{Z}$ is surjective, by \cite[Theorem 9.2]{Gabber-Liu-Lorenzini}, there exists $z \in Z_0(X)$ such that $\deg(z) = 1$ and such that for every $x$ in the support of $z$, the extension $k(x)/k$ is separable.  Therefore, by Lemma \ref{curve_tame_special_compatibility}, we have \[\langle z_K, \alpha_K \rangle = \langle z, \alpha\rangle_K.\] Since $H^0_S(X_K) \rightarrow \mathbb{Z}$ is injective, and $\deg(z_K)=1$, we have $z_K = \eta$. Therefore, we have \[\langle z_K,\alpha_K\rangle = \langle \eta, \alpha_K\rangle = \alpha,\] where the last equality follows from Lemma \ref{curve_tame_generic_compatibility}. Therefore, $\alpha = \langle z,\alpha\rangle_K$, which lies in the image of $H^1(k,\mathcal{G}) \rightarrow H^1(k(X),\mathcal{G})$.
\end{proof}

In the case of smooth proper varieties, we get the following result.

\begin{corollary}
Let $X$ be a smooth proper variety over $k$ such that for any field extension $K$ of $k$, the degree map $CH_0(X_K) \rightarrow \mathbb{Z}$ is an isomorphism. Let $\mathcal{G}$ be as in Proposition \ref{well_definedness_of_pairing}. Then the natural map $H^1(k, \mathcal{G}) \rightarrow H^1_{\mathrm{ur}}(k(X)/k,\mathcal{G})$ is an isomorphism.
\end{corollary}

\begin{proof}
Since $X$ is a smooth proper variety, the natural map $H_0^S(X_K) \rightarrow CH_0(X_K)$ is an isomorphism for any field extension $K$. The desired isomorphism follows immediately from Theorem \ref{curve_tame_equals_cohomology_of_base_field} and Proposition \ref{proper_equality}.
\end{proof}

\subsection*{Acknowledgements}

We thank Kay R\"ulling for pointing out a gap in the previous version of this article and for very helpful discussions and email correspondences.  We also thank Federico Binda, Amit Hogadi and Alberto Merici for their helpful comments and discussions.

\end{document}